\documentclass{amsart}

\usepackage{graphicx}
\usepackage{epsfig}

\newtheorem{theorem}{Theorem}[section]
\newtheorem{lemma}[theorem]{Lemma}

\theoremstyle{definition}

\newtheorem{corollary}[theorem]{Corollary}

\theoremstyle{remark}
\newtheorem{remark}[theorem]{Remark}

\numberwithin{equation}{section}

\usepackage{hyperref}

\begin{document}

\title[An iterative scheme for solving NOE]{An iterative scheme for solving equations
with locally $\sigma$-inverse monotone operators}

\author{N. S. Hoang}

\address{Mathematics Department, Kansas State University,
Manhattan, KS 66506-2602, USA}

\email{nguyenhs@math.ksu.edu}

%


\date{}


\begin{abstract}
An iterative scheme for solving ill-posed
nonlinear equations with locally $\sigma$-inverse monotone operators is studied in this paper.
A stopping rule of discrepancy type is proposed.  
The existence of $u_{n_\delta}$ satisfying the proposed stopping rule is proved. 
The convergence of this element to the minimal-norm solution is justified
mathematically. 
\end{abstract}

\maketitle

\section{Introduction}

In this paper we study an iterative scheme 
for solving the equation
\begin{equation}
\label{aeq1}
F(u)=f,
\end{equation}
 where
$F$ is a locally $\sigma$-inverse monotone operator in a 
Hilbert space $H$, and equation \eqref{aeq1} is assumed solvable, 
possibly nonuniquely. 
An operator $F$ is called locally $\sigma$-inverse monotone if 
for any $R>0$ there exists a constant $\sigma_R>0$ such that
\begin{equation}
\label{ceq2}
\langle F(u)-F(v),u-v\rangle\ge \sigma_R\|F(u)-F(v)\|^2,\qquad \forall u,v\in B(0,R)\subset H.
\end{equation}
Here, $\langle \cdot,\cdot\rangle$ denotes the inner product in $H$.
If the constant $\sigma_R$ in \eqref{ceq2} is independent of $R$ then we call
$F$ a $\sigma$-inverse monotone operator. 

A necessary condition for an operator $F$ to be
$\sigma$-inverse
monotone is the following:
$$
\|F(u)-F(v)\|\leq \sigma^{-1}\|u-v\|.
$$
Indeed, inequality \eqref{ceq2} and the Cauchy inequality imply the
above estimate.
If the $\sigma$-inverse monotone operator is a homeomorphism,
then
its inverse is strongly monotone:
$$
\langle F^{-1}(u)-F^{-1}(v),u-v \rangle\geq \sigma \|u-v\|^2.
$$
An example of $\sigma$-inverse operator is a linear
selfadjoint compact nonnegative-definite operator $A$.
Indeed, if $\lambda_1\geq \lambda_2\geq....\geq 0$ are its
eigenvalues and
$\phi_j$ are the 
corresponding normalized eigenvectors, then
$$
\langle Au-Av,u-v \rangle=\sum_{j}\lambda_j|\langle u-v,\phi_j\rangle |^2\geq
\sigma \sum_{j}\lambda_j^2|\langle u-v,\phi_j\rangle |^2=\sigma
\|Au-Av\|^2,
$$
where $\sigma=\lambda_1^{-1}$. 
An example of locally $\sigma$-inverse monotone operator is
a nonlinear Fr\'{e}chet differentiable monotone operator $F:H\to H$ provided that $H$ is a complex Hilbert space
and $F'$ is locally bounded, i.e., for any $R>0$ there exists a constant $M(R)$ 
such that 
$$
\|F'(u)\| \le M(R),\qquad \forall u\in B(0,R) 
$$
(see Lemma~\ref{lemma2.10} in Section 2). 
In Lemma~\ref{lemma2.10} we also prove that 
if $H$ is a real Hilbert space, $F:H\to H$ is a Fr\'{e}chet differentiable 
monotone operator and $F'$
is a selfadjoint locally bounded operator, then $F$ is also a locally $\sigma$-inverse monotone operator. 
If \eqref{ceq2} holds, then the operator $\sigma_R F$ satisfies \eqref{ceq2} with $\sigma_R = 1$. 

It is clear that if $F$ is $\sigma$-inverse monotone, then it is monotone, i.e.,
\begin{equation}
\label{ceq2.1}
\langle F(u)-F(v),u-v\rangle\ge 0,\qquad \forall u,v\in H.
\end{equation}
It is known (see, e.g., \cite{R499}), that the set 
$\mathcal{N}:=\{u:F(u)=f\}$ is closed and convex if 
$F$ is monotone and continuous. 
A closed and convex set in a Hilbert space has a unique 
minimal-norm element. This element in 
$\mathcal{N}$ we denote by $y$, $F(y)=f$, and call it the minimal-norm 
solution to equation (1). 
We assume that 
 $f=F(y)$ is not known but
 $f_\delta$, the noisy data, are known, and $\|f_\delta-f\|\le \delta$. If $F'(u)$
 is not boundedly invertible then 
 solving equation \eqref{aeq1} for $u$ given noisy data $f_\delta$ is 
often (but not always) an ill-posed problem.
When $F$ is a linear bounded operator many methods were proposed for solving stably 
equation \eqref{aeq1} (see \cite{I}--\cite{R499} and the references therein). 
However, when $F$ is nonlinear then the theory is
less complete. 

Methods for solving equation \eqref{aeq1} were extensively studied in \cite{546}--\cite{R544}, 
\cite{R499}--\cite{491}. 
In \cite{R499}, \cite{546}, the following iterative scheme for solving equation \eqref{aeq1}
with monotone operators $F$ was investigated:
\begin{equation}
\label{hichehic}
u_{n+1} = u_n - \big{(}F'(u_n) + 
a_nI\big{)}^{-1}\big{(}F(u_n)+a_nu_n - f_\delta\big{)},
\qquad u_0=\tilde{u}_0.
\end{equation}
The convergence of this method was justified with an {\it a apriori} 
and an {\it a posteriori}
choice of stopping rules (see \cite{546}). 
In \cite{R544} a continuous version of the regularized Newton method \eqref{hichehic} with a stopping rule of discrepancy type is 
formulated and justified. 
Another iterative scheme with an {\it a posteriori}
choice of stopping rule was formulated and 
justified in \cite{R549}. 

In this paper we consider the following iterative for 
a stable solution to equation \eqref{aeq1}:
\begin{equation}
\label{aeq2}
u_{n+1} = u_n - \gamma_n \big{[}F(u_n)+a_nu_n - f_\delta\big{]},
\qquad u_0=\tilde{u}_0,
\end{equation}
where $F$ is a locally $\sigma$-inverse monotone operator, $\gamma_n\in(0,1)$, $n\ge 0$, and $\tilde{u}_0$ 
is a suitably chosen element in $H$ which will be specified later.

The advantages of this iterative scheme compared with
\eqref{hichehic} are:
\begin{enumerate}
\item{the absence of the inverse operator in the algorithm, which
makes the algorithm \eqref{aeq2} less expensive than \eqref{hichehic}}
\item{one does not have to compute 
the Fr\'{e}chet derivative of $F$}
\item{the Fr\'{e}chet differentiability of $F$ is not required.}
\end{enumerate}
A more expensive algorithm \eqref{hichehic} may converge faster than  
\eqref{aeq2} in some cases.

The convergence of the method \eqref{aeq2} for exact data was proved in
\cite{R499}. For noisy data it was proved in \cite{R550} that the element
$u_{n_\delta}$, defined by \eqref{aeq2} and an {\it a posteriori} choice
of stopping rule, converges to a solution to \eqref{aeq1} when $u_0$ and
$a_n$ are suitably chosen, provided that $H$ is a complex Hilbert space
and $F$ is a Fr\'{e}chet differentiable monotone operator.  However, it is
of interest to prove the convergence to the minimal-norm solution to
\eqref{aeq1}. The minimal-norm solution in problems with a linear
operator $F$ is the solution orthogonal to the null-space of $F$.
In linear algebra it is called the normal solution, and it is the
solution that is of interest in many computational problems.   
In nonlinear problems the minimal-norm solution is also the solution of 
interest in many cases, because it is often the solution with minimal 
energy.

In this paper we investigate a stopping rule based on a {\it discrepancy 
principle} (DP) for the iterative scheme \eqref{aeq2}. 
Using the local $\sigma$-inverse monotonicity of $F$, we prove convergence of
the method \eqref{aeq2} to the minimal-norm solution to \eqref{aeq1}. 
The rate of decay of the regularizing sequence $a_n$ in this paper 
is also faster than the one in \cite{R550}.
This saves the computer time and results in a faster 
convergence of our method. 
The main results of this paper are Theorem~\ref{mainthm}, \ref{theorem2}, and \ref{theorem3}. 
In Theorem~\ref{mainthm} 
a DP is formulated and   
the existence of a stopping time $n_\delta$ is proved. 
The convergence of the iterative scheme with the proposed DP to a solution to \eqref{aeq1} is proved in 
Theorem~\eqref{theorem2}. In Theorem~\eqref{theorem3} sufficient conditions 
for the convergence of the iterative scheme with the proposed DP 
to the minimal-norm solution to \eqref{aeq1} is justified mathematically.

\section{Auxiliary results}

Let us consider the following equation:
\begin{equation}
\label{2eq2}
F(\tilde{V}_{a,\delta})+a \tilde{V}_{a,\delta}-f_\delta = 0,\qquad a=const>0.
\end{equation}
It is known (see, e.g., \cite{D} and \cite{R499}) that equation 
\eqref{2eq2} 
with monotone continuous operator $F$ has a unique solution for any 
fixed $a>0$ and any $f_\delta\in H$. 

Lemmas \ref{lemma1}, \ref{lemma3} and \ref{lemma5} can be found in \cite{546}.
We include the proofs for the convenience of the reader. 

\begin{lemma}
\label{lemma1}
If \eqref{ceq2} holds and $F$ is continuous, then
$\|\tilde{V}_{a,\delta}\|=O(\frac{1}{a})$ as $a\to\infty$, and
\begin{equation}
\label{4eq2}
\lim_{a\to\infty}\|F(\tilde{V}_{a,\delta})-f_\delta\|=\|F(0)-f_\delta\|.
\end{equation}
\end{lemma}

\begin{proof}
Rewrite \eqref{2eq2} as
$$
F(\tilde{V}_{a,\delta}) - F(0) + a\tilde{V}_{a,\delta} + F(0)-f_\delta = 0.
$$
Multiply this equation
by $\tilde{V}_{a,\delta}$, use the inequality $\langle F(\tilde{V}_{a,\delta})-F(0),\tilde{V}_{a,\delta}-0\rangle \ge 0$, which
follows from \eqref{ceq2}, and get:
$$
a\|\tilde{V}_{a,\delta}\|^2\le\langle a\tilde{V}_{a,\delta} + F(\tilde{V}_{a,\delta})-F(0), \tilde{V}_{a,\delta}\rangle = 
\langle f_\delta-F(0), \tilde{V}_{a,\delta}\rangle \le \|f_\delta-F(0)\|\|\tilde{V}_{a,\delta}\|.
$$
Therefore,
$\|\tilde{V}_{a,\delta}\|=O(\frac{1}{a})$. This and the continuity of $F$ imply \eqref{4eq2}.
\end{proof}

Let us recall the following result (see Lemma 6.1.7 \cite[p. 112]{R499}):
\begin{lemma}
\label{rejectedlem}
Assume that equation \eqref{aeq1} is solvable. Let $y$ be its minimal-norm solution.
Assume that $F:H\to H$ is a continuous monotone operator. Then
$$
\lim_{a\to 0} \|\tilde{V}_{a}-y\| = 0,
$$
where $\tilde{V}_a:=\tilde{V}_{a,0}$ which solves \eqref{2eq2} with $\delta=0$.
\end{lemma}

Let us consider the following equation
\begin{equation}
\label{3eq2}
F(V_{\delta,n})+a_n V_{\delta,n}-f_\delta = 0,\qquad a_n>0,
\end{equation}

For simplicity let us denote $V_n:=V_{\delta,n}$ when $\delta\not=0$.

\begin{lemma}
\label{lemma3} 
Assume that $0<(a_n)_{n=0}^\infty \searrow 0$. 
Then 
\begin{equation}
\label{eeq5}
\lim_{n\to\infty}\|F(V_{n})-f_\delta\|\le \delta.
\end{equation}
\end{lemma}

\begin{proof}
We have $F(y)=f$, and
\begin{align*}
0=&\langle F(V_n)+a_nV_n-f_\delta, F(V_n)-f_\delta \rangle\\
=&\|F(V_n)-f_\delta\|^2+a_n\langle V_n-y, F(V_n)-f_\delta \rangle + a_n\langle y, F(V_n)-f_\delta \rangle\\
=&\|F(V_n)-f_\delta\|^2+a_n\langle V_n-y, F(V_n)-F(y) \rangle + a_n\langle V_n-y, f-f_\delta \rangle \\
&+ a_n\langle y, F(V_n)-f_\delta \rangle\\
\ge &\|F(V_n)-f_\delta\|^2 + a_n\langle V_n-y, f-f_\delta \rangle + a_n\langle y, F(V_n)-f_\delta \rangle.
\end{align*}
Here the inequality $\langle V_n-y, F(V_n)-F(y) \rangle\ge0$ was used. 
Therefore
\begin{equation}
\label{1eq1}
\begin{split}
\|F(V_n)-f_\delta\|^2 &\le -a_n\langle V_n-y, f-f_\delta \rangle - a_n\langle y, F(V_n)-f_\delta \rangle\\
&\le a_n\|V_n-y\| \|f-f_\delta\| + a_n\|y\| \|F(V_n)-f_\delta\|\\
&\le  a_n\delta \|V_n-y\|  + a_n\|y\| \|F(V_n)-f_\delta\|.
\end{split}
\end{equation}
On the other hand, one has:
\begin{align*}
0&= \langle F(V_n)-F(y) + a_nV_n +f -f_\delta, V_n-y\rangle\\
&=\langle F(V_n)-F(y),V_n-y\rangle + a_n\| V_n-y\| ^2 + a_n\langle y, V_n-y\rangle + \langle f-f_\delta, V_n-y\rangle\\
&\ge  a_n\| V_n-y\| ^2 + a_n\langle y, V_n-y\rangle + \langle f-f_\delta, V_n-y\rangle,
\end{align*}
where the inequality $\langle V_n-y, F(V_n)-F(y) \rangle\ge0$ was used. Therefore,
$$
a_n\|V_n-y\|^2 \le a_n\|y\|\|V_n-y\|+\delta\|V_n-y\|.
$$
This implies
\begin{equation}
\label{1eq2}
a_n\|V_n-y\|\le a_n\|y\|+\delta.
\end{equation}
From \eqref{1eq1} and \eqref{1eq2}, and an elementary inequality $ab\le \epsilon a^2+\frac{b^2}{4\epsilon},\,\forall\epsilon>0$, one gets:
\begin{equation}
\label{3eq4}
\begin{split}
\|F(V_n)-f_\delta\|^2&\le \delta^2 + a_n\|y\|\delta + a_n\|y\| \|F(V_n)-f_\delta\|\\
&\le \delta^2 + a_n\|y\|\delta + \epsilon \|F(V_n)-f_\delta\|^2 + 
\frac{1}{4\epsilon}a_n^2\|y\|^2,
\end{split}
\end{equation}
where $\epsilon>0$ is fixed, independent of $n$, and can be chosen 
arbitrary small. 
Let $n\to\infty$ so $a_n\searrow 0$. Then \eqref{3eq4} implies
$\lim_{n\to\infty}(1-\epsilon)\|F(V_n)-f_\delta\|^2\le \delta^2$,\, $\forall\, \epsilon>0$. This implies
$\lim_{n\to\infty}\|F(V_n)-f_\delta\| \le \delta$.

Lemma~\ref{lemma3} is proved.
\end{proof}

\begin{remark}
\label{rem3}
{\rm 
Let $V_{0,n}:=V_{\delta,n}|_{\delta=0}$. Then $F(V_{0,n})+a_n V_{0,n}-f=0$. 
Note that
\begin{equation}
\label{rejected11}
\|V_{\delta,n}-V_{0,n}\|\le \frac{\delta}{a_n}.
\end{equation}
Indeed, from \eqref{2eq2} one gets
$$
F(V_{\delta,n}) - F(V_{0,n}) + a_n (V_{\delta,n}-V_{0,n}) = f_\delta - f.
$$
Multiply this equality with $V_{\delta,n}-V_{0,n}$ and use \eqref{ceq2} to get:
\begin{align*}
\delta \|V_{\delta,n}-V_{0,n}\| &\ge \langle f_\delta - f, V_{\delta,n}-V_{0,n} \rangle\\
&= \langle F(V_{\delta},n) - F(V_{0,n}) + a_n (V_{\delta,n}-V_{0,n}), V_{\delta,n}-V_{0,n} \rangle\\
&\ge a_n \|V_{\delta,n}-V_{0,n}\|^2.
\end{align*}
This implies \eqref{rejected11}. 
Similarly, from the equation
$$
F(V_{0,n}) + a_n V_{0,n} -F(y)=0,
$$
one derives that
\begin{equation}
\label{rejected12}
\|V_{0,n}\| \le \|y\|.
\end{equation}
Similar arguments one can find in \cite{R499}. 

From \eqref{rejected11} and \eqref{rejected12}, one gets the following estimate:
\begin{equation}
\label{2eq1}
\|V_{n}\|\le \|V_{0,n}\|+\frac{\delta}{a_n}\le \|y\|+\frac{\delta}{a_n}.
\end{equation}

From equation \eqref{3eq2} one gets
$$
F(V_{n+1}) - F(V_n) = a_n V_n - a_{n+1}V_{n+1}. 
$$
This and the monotonicity of $F$ imply 
\begin{equation}
\label{eqjkl6}
\begin{split}
0 &\le \langle a_n V_n - a_{n+1}V_{n+1}, V_{n+1} - V_n \rangle \\
&= - a_n \|V_n - V_{n+1}\|^2 + (a_n - a_{n+1})\langle V_{n+1}, V_{n+1} - V_n \rangle\\
&\le - a_n \|V_n - V_{n+1}\|^2 + (a_n - a_{n+1})\|V_{n+1}\|\| V_{n+1} - V_n\|.
\end{split}
\end{equation}
Thus, one gets
\begin{equation}
\label{eqjkl7}
\|V_n - V_{n+1}\| \le \frac{a_n - a_{n+1}}{a_n}\|V_{n+1}\|,\qquad \forall n\ge 0.
\end{equation}
}
\end{remark}

\begin{lemma}
\label{lemma5}
Assume $\|F(0)-f_\delta\|>0$. 
Let $0<a_n\searrow 0$, $F$ be monotone, and
$$
\ell_n:=\|F(V_n) - f_\delta\|,\quad  k_n:=\|V_n\|,\qquad n=0,1,...,
$$ 
where $V_n$ solves \eqref{3eq2}. 
Then
$\ell_n$ is decreasing and $k_n$ is increasing.
\end{lemma}
 
\begin{proof}
Since $\|F(0)-f_\delta\|>0$, it follows that $k_n\not=0,\, \forall n\ge0$. 
One has
\begin{equation}
\label{1eq3}
\begin{split}
0&\le \langle F(V_n)-F(V_m),V_n-V_m\rangle\\
&= \langle -a_nV_n + a_mV_m,V_n-V_m\rangle\\
&= (a_n+ a_m)\langle V_n,V_m \rangle -a_n\|V_n\|^2 -  a_m\|V_m\|^2.
\end{split}
\end{equation}
Thus,
\begin{equation}
\label{2eq6}
\begin{split}
0&\le (a_n+ a_m)\langle V_n,V_m \rangle -a_n\|V_n\|^2 -  a_m\|V_m\|^2\\
& \le  (a_n+ a_m)\|V_n\|\|V_m \| - a_n\|V_n\|^2 -  a_m\|V_m\|^2\\
& = (a_n \|V_n\| -  a_m \|V_m\|)(\|V_m\|-\|V_n\|)\\
& = (\ell_n-\ell_m)(k_m-k_n).
\end{split}
\end{equation}

From \eqref{3eq2} one gets 
\begin{equation}
\label{eq*}
\ell_n= \|F(V_n) - f_\delta\|=a_n\|V_n\| = a_n k_n,\qquad n\ge 0.
\end{equation}
If $k_m> k_n$ then \eqref{2eq6} and \eqref{eq*} imply $\ell_n\ge \ell_m$, so
$$
a_nk_n=\ell_n \ge \ell_m = a_mk_m>  a_m k_n.
$$
Thus, if $k_m> k_n$ then $ a_m< a_n$ and, therefore, $m> n$,
because $a_n$ is decreasing.

Similarly, if $k_m< k_n$ then $\ell_n\le \ell_m$.
This implies $ a_m> a_n$, so $m< n$.

If $k_m= k_n$ then \eqref{1eq3} implies
$$
\|V_m\|^2\le \langle V_m,V_n \rangle \le \|V_m\|\|V_n\| = \|V_m\|^2.
$$
This implies $V_m=V_n$, and then $a_n=a_m$. Hence, $m=n$, because $a_n$ is decreasing.

Therefore $\ell_n$ is decreasing
and $k_n$ is increasing. Lemma~\ref{lemma5} is proved.
\end{proof}

\begin{remark}
\label{remmoi}
{\rm
From Lemma~\ref{lemma1} and Lemma~\ref{lemma5} one concludes that
\begin{equation}
\label{eqstar}
a_n\|V_n\|=\|F(V_n)-f_\delta\| \le \|F(0)-f_\delta\|,\qquad \forall n\ge 0.
\end{equation}
}
\end{remark}

Let
$0<a(t)\in C^1(\mathbb{R}_+)$  
satisfy the following conditions:
\begin{equation}
\label{eqsxa}
0< a(t)\searrow 0,
\qquad \nu(t):=\frac{|\dot{a}(t)|}{a^2(t)}\searrow 0,\qquad t\ge 0.
\end{equation}
Let $0<h=const$ and 
$$
a_n:=a(nh),\qquad n\ge 0.
$$

\begin{remark}
It follows from \eqref{eqsxa} that
\begin{equation}
\label{eqz8}
0<\frac{1}{a_{n+1}}-\frac{1}{a_n}=-\int_{nh}^{(n+1)h}\frac{\dot{a}(s)}{a^2(s)}ds \le h\nu(n)\le h\nu(0).
\end{equation}
Inequalities \eqref{eqz8} imply
\begin{equation}
\label{eql2}
1< \frac{a_n}{a_{n+1}}\le 1+ a_nh\nu(0).
\end{equation}
From the relation $\lim_{n\to\infty}a_n = 0$ and \eqref{eql2} one gets
\begin{equation}
\label{eqvv}
\lim_{n\to\infty} \frac{a_n}{a_{n+1}} = 1.
\end{equation}
From \eqref{eqsxa} and \eqref{eqz8} one gets
\begin{equation}
\label{eqqe}
\lim_{n\to\infty} \frac{a_n - a_{n+1}}{a_na_{n+1}}= 0.
\end{equation}
\end{remark}

\begin{remark}
Let $b\in(0,1)$, $c\ge1$, $d>0$ and 
$$
a(t)=\frac{d}{(c+t)^b}.
$$
Then $a(t)$ satisfies \eqref{eqsxa}.
\end{remark}

\begin{lemma}
\label{lemauxi2}
Let $0<h=const$ and $a(t)$ satisfy \eqref{eqsxa} and the following conditions
\begin{equation}
\label{eqkkk2}
a(0)h\le 2, \qquad \nu(0)=\frac{|\dot{a}(0)|}{a^2(0)}\le \frac{1}{10}. 
\end{equation}
Let $a_n:=a(nh)$ and 
\begin{equation}
\label{eqxvn}
\varphi_n:=\sum_{i=1}^{n} \frac{a_{i}h}{2},\qquad n\ge 1,
\end{equation}
Then the following inequality holds
\begin{equation}
\label{auxieq3}
e^{-\varphi_n}\sum_{i=0}^{n-1} e^{\varphi_{i+1}}(a_i-a_{i+1})\|V_i\| \le
 \frac{1}{2}a_n\|V_n\|,\qquad n\ge 1.
\end{equation} 
\end{lemma}

\begin{proof} 
First, let us prove that
\begin{equation}
\label{eqrj1}
e^{\varphi_n}(a_{n-1} - a_n) \le \frac{1}{2}(a_ne^{\varphi_n} - a_{n-1}e^{\varphi_{n-1}}),
\qquad \forall n\ge 1.
\end{equation}
Inequality \eqref{eqrj1} is equivalent to
\begin{equation}
\label{eqrj2}
\frac{3a_n}{a_{n-1}} \ge \frac{2e^{\varphi_n}+e^{\varphi_{n-1}}}{e^{\varphi_n}}
=2+e^{- \frac{ha_n}{2}},\qquad n\ge 1.
\end{equation}
This inequality is equivalent to
\begin{equation}
\label{eqrj5}
\frac{a_{n-1}-a_n}{a_{n-1}a_n} \le \frac{1-e^{-\frac{ha_n}{2}}}{3a_n},\qquad \forall n\ge 1. 
\end{equation}
Let us prove \eqref{eqrj5}. From \eqref{eqsxa} and \eqref{eqkkk2} one gets
\begin{equation}
\label{eqrj6}
\frac{a_{n-1}-a_n}{a_{n-1}a_n} = \int_{(n-1)h}^{nh}\frac{|\dot{a}(s)|}{a^2(s)}ds \le h\nu((n-1)h) \le h\nu(0)
\le \frac{h}{10}. 
\end{equation}

Note that the function $\tilde{f}(x) = \frac{1-e^{-x}}{x}$ is decreasing on $(0,\infty)$. 
Therefore, one gets
\begin{equation}
\label{eqrj7}
\frac{1-e^{-\frac{ha_n}{2}}}{3a_n} = 
\frac{h \tilde{f}(\frac{ha_n}{2})}{6} \ge \frac{h \tilde{f}(\frac{ha_0}{2})}{6} 
\ge \frac{h \tilde{f}(1)}{6}\ge \frac{h\frac{6}{10}}{6} = \frac{h}{10}. 
\end{equation}
We have used the inequalities $a_nh\le a_0h\le 2$, $\forall n\ge 1$, and $\tilde{f}(1)>\frac{6}{10}$ in \eqref{eqrj7}. 
Inequality \eqref{eqrj5} follows from \eqref{eqrj6} and \eqref{eqrj7}. Thus, \eqref{eqrj1} holds.

From inequality \eqref{eqrj1} one obtains
\begin{equation}
\label{eqpp2}
2\sum_{i=0}^{n-1} e^{\varphi_{i+1}}(a_i-a_{i+1})\le 
\sum_{i=0}^{n-1} (a_{i+1}e^{\varphi_{i+1}} - a_{i}e^{\varphi_{i}})
 < e^{\varphi_n}a_n,\qquad n\ge 1.
\end{equation}
Multiplying \eqref{eqpp2} by $\frac{1}{2}\|V_n\|e^{-\varphi_n}$ and recalling the fact that $\|V_i\|$
is increasing (see Lemma~\ref{lemma5}), one gets inequality
\eqref{auxieq3}. Lemma~\ref{lemauxi2} is proved.
\end{proof}

\begin{lemma}
\label{lemma2}
Let $R$ and $\sigma_R$ be positive constants and $F$ be 
an operator in a Hilbert space $H$
satisfying the following inequality:
\begin{equation}
\label{eqg3}
\langle F(u)-F(v),u-v\rangle \ge \sigma_R \|F(u)-F(v)\|^2,\qquad \forall u,v\in B(0,R)\subset H.
\end{equation}
Assume that
\begin{equation}
\label{eqg1}
0<\gamma\le \frac{2}{\sigma^{-1}_R+2a},\qquad a=const\ge 0.
\end{equation}
Then
\begin{equation}
\label{eqg2}
\mu(u,v):=\|u-v - \frac{\gamma}{1-\gamma a}[F(u)-F(v)]\| \le \|u-v\|,\qquad \forall u,v\in B(0,R).
\end{equation}
\end{lemma}

\begin{proof}
Let us fix $R>0$ and denote $\sigma:=\sigma_R$ and $w:=u-v$. 
From \eqref{eqg3}, one gets, $\forall u,v\in B(0,R)$, the following inequality:
\begin{equation}
\label{eqg4}
\begin{split}
\mu^2(u,v) &= \|w\|^2 - \frac{2\gamma}{1-\gamma a}\langle w, F(u)-F(v)\rangle + 
\frac{\gamma^2}{(1-\gamma a)^2}\|F(u)-F(v)\|^2\\
&\le \|w\|^2 - \frac{2\gamma}{1-\gamma a} \sigma\|F(u)-F(v)\|^2 + 
\frac{\gamma^2}{(1-\gamma a)^2}\|F(u)-F(v)\|^2\\
&= \|w\|^2 - \bigg{(}\frac{2\gamma\sigma}{1-\gamma a} - \frac{\gamma^2}{(1-\gamma a)^2} \bigg{)}\|F(u)-F(v)\|^2. 
\end{split}
\end{equation}
It follows from \eqref{eqg1} that
\begin{equation}
\label{eqg5}
\begin{split}
\frac{2\gamma\sigma}{1-\gamma a} - \frac{\gamma^2}{(1-\gamma a)^2} 
&= \frac{\gamma\sigma}{(1-\gamma a)^2}[2(1-\gamma a) -\sigma^{-1}\gamma]\\
&= \frac{\gamma\sigma}{(1-\gamma a)^2}(\sigma^{-1}+2a)[\frac{2}{\sigma^{-1}+2a} -\gamma]\ge 0.
\end{split}
\end{equation}
Inequality \eqref{eqg2} follows from inequalities \eqref{eqg4} and \eqref{eqg5}.
Lemma~\ref{lemma2} is proved. 
\end{proof}

\begin{lemma}
\label{lemma2.10}
Let $F:H\to H$ be a Fr\'{e}chet differentiable monotone operator with locally bounded $F'$, i.e.,
\begin{equation}
\label{eqsoccon}
\|F'(u)\| \le M(R),\qquad \forall u\in B(u_0,R),
\end{equation}
where $H$ is a Hilbert space. 
Let one of the following assumptions hold:
\begin{enumerate}
\item{$H$ is a real Hilbert space and $F'$ is selfadjoint,} 
\item{$H$ is a complex Hilbert space.}
\end{enumerate} 
Then $F$ is a locally $\sigma$-inverse monotone operator, i.e.,
for all $R>0$ there exists $\sigma_R>0$ such that
\begin{equation}
\label{eqo3}
\langle F(u)-F(v), u-v\rangle \ge \sigma_R\|F(u)-F(v)\|^2,\qquad \forall u,v\in B(0,R).
\end{equation}
Moreover,
\begin{equation}
\label{eqsigma}
\sigma_R = \frac{1}{M(R)},\qquad R>0.
\end{equation}
\end{lemma}

\begin{proof}
Fix $u,v\in B(0,R)$. One has
\begin{equation}
\label{eqo4}
F(u) - F(v) =J(u-v),\qquad J:= \int_0^1 F'(v+\xi (u-v))d\xi.
\end{equation}
By our assumption $J$ is a selfadjoint operator and 
\begin{equation}
\label{eqo5}
0\le J \le M(R),\qquad M(R):=\sup_{w\in B(0,R)} \|F'(w)\|.
\end{equation}
This and the selfadjointness of $J$ imply
\begin{equation}
\label{eqo6}
0\le J(I -\sigma_R J) = (I -\sigma_R J)J,,
\end{equation}
where $I$ is the identity operator in $H$ and $\sigma_R$ is defined by \eqref{eqsigma}. 
Thus,
\begin{equation}
\label{eqo7}
\begin{split}
\langle F(u)-F(v), u-v\rangle &= \langle J(u-v), u-v\rangle\\
&= \langle J(u-v),(I-\sigma_R J) (u-v)\rangle + \sigma_R \| J(u-v)\|^2\\
&= \langle \big{[}(I-\sigma_R J)J\big{]}(u-v), (u-v)\rangle + \sigma_R \| J(u-v)\|^2\\ 
&\ge \sigma_R \| J(u-v)\|^2 = \sigma_R \|F(u)-F(v)\|^2.
\end{split}
\end{equation}
This implies \eqref{eqo3}. Lemma~\ref{lemma2.10} is proved. 
\end{proof}

\begin{remark}
It follows from the proof of Lemma~\ref{lemma2.10} that if $F'(u)$ 
is self-addjoint and uniformly bounded, i.e., the constant $M=M(R)$ in \eqref{eqsoccon} is independent of $R$, 
 then $F$ is a $\sigma$-inverse monotone operator with $\sigma = \frac{1}{M}$.  
\end{remark}

\begin{lemma}
\label{lemma10}
Let $0<h=const$, $a(t)$ satisfy \eqref{eqsxa}, $a_n:=a(nh)$, and   
\begin{equation}
\label{eqkkk}
\phi_n = h\sum_{i=0}^n a_i, \qquad \phi(t):=\int_0^t a(s)ds.
\end{equation}
Then
\begin{align}
\label{eqbbb3xxxx}
e^{-\phi_{n-1}}\sum_{i=0}^{n-1}e^{\phi_{i}}(a_i - a_{i+1})
&\le e^{a(0)h}e^{-\phi(nh)}\int_0^{nh} e^{\phi(s)}|\dot{a}(s)|ds ,\qquad \forall n\ge 1,\\
\label{eqbbb3}
e^{-\phi_{n-1}}\sum_{i=0}^{n-1}e^{\phi_{i}}\frac{a_i - a_{i+1}}{a_i} 
&\le e^{a(0)h}e^{-\phi(nh)}\int_0^{nh} e^{\phi(s)}\frac{|\dot{a}(s)|}{a(s)}ds ,\qquad \forall n\ge 1.
\end{align}
\end{lemma}

\begin{proof}
{\it Let us prove \eqref{eqbbb3}. Inequality \eqref{eqbbb3xxxx} is obtained similarly.}

Since $a_n=a(nh)$ and $0<a(t)\searrow 0$, one gets
\begin{equation}
\label{eqf10}
\begin{split}
\phi_n - \phi_{i} &= \sum_{k=i+1}^{n}a_{k}h \ge \sum_{k=i+1}^{n}\int_{kh}^{(k+1)h}a(s)ds \\
&= \int_{(i+1)h}^{(n+1)h} a(s)ds
= \phi((n+1)h)-\phi((i+1)h),\qquad  0\le i\le n.
\end{split}
\end{equation}
This and the inequalities
$$
\phi((i+1)h) - \phi(s)  = \int_s^{(i+1)h}a(s)ds
\le \int_s^{(i+1)h}a(0)ds \le a(0)h,
$$
for all $s\in [ih,(i+1)h]$, imply
\begin{equation}
\label{eqnn2}
\begin{split}
-\phi_{n-1} + \phi_{i} &\le -\phi(nh) + \phi(s) +a(0)h,\qquad \forall s \in [ih,(i+1)h],
\end{split}
\end{equation}
where $0\le i\le n-1$. 

Since $0<a_n\searrow 0$ and $|\dot{a}(t)| = -a(t)$, one obtains
\begin{equation}
\label{eqnn1}
\frac{a_i-a_{i+1}}{a_i}
= \int_{ih}^{(i+1)h} \frac{|\dot{a}(s)|}{a_i}ds
\le \int_{ih}^{(i+1)h} \frac{|\dot{a}(s)|}{a(s)}ds.
\end{equation}
It follows from \eqref{eqnn1} and \eqref{eqnn2} that
\begin{equation}
\label{eqnn3}
\begin{split}
e^{-\phi_{n-1}}\sum_{i=0}^{n-1}e^{\phi_{i}}\frac{a_i - a_{i+1}}{a_i} 
&\le \sum_{i=0}^{n-1}\int_{ih}^{(i+1)h}e^{-\phi_{n-1}+\phi_i} \frac{|\dot{a}(s)|}{a(s)}ds\\
&\le e^{a(0)h} e^{-\phi(nh)}\int_0^{nh} e^{\phi(s)}\frac{|\dot{a}(s)|}{a(s)}ds,\qquad \forall n\ge 1.
\end{split}
\end{equation}
Lemma~\ref{lemma10} is proved.
\end{proof}

\begin{lemma}
\label{lemma11}
Let $0<h=const$, $a(t)$ satisfy \eqref{eqsxa}, $a_n=a(nh)$ and $\phi_n$ be as in \eqref{eqkkk}. Then
\begin{equation}
\label{eql6}
\lim_{n\to \infty} e^{\phi_{n-1}}a_n = \infty,
\end{equation}
and
\begin{equation}
\label{eql9}
M:=\lim_{n\to\infty} \frac{\sum_{i=0}^{n} e^{\phi_{i}}(a_i - a_{i+1})}{e^{\phi_{n}}a_{n+1}} = 0.
\end{equation}
\end{lemma}

\begin{proof}
{\it Let us first prove \eqref{eql6}.}

From \eqref{eqsxa} and \eqref{eqkkk}, one gets
\begin{equation}
\label{eqsxaz}
\lim_{n\to\infty} \phi_n =  \lim_{n\to\infty} h\sum_{i=0}^n a_i \ge \int_0^\infty a(s)ds 
\ge \int_0^\infty  \frac{1}{\nu(0)} \frac{-\dot{a}(s)}{a(s)}ds=\infty.
\end{equation}
We claim that if $n>0$ is sufficiently large, then the following inequality holds:
\begin{equation}
\label{xex5}
\phi_{n-1} \ge  \ln \frac{1}{a^2_n}.
\end{equation}
Indeed, using a discrete analog of L'Hospital's rule, the relation  
$\ln(1+x) = x + o(x)$, and \eqref{eqqe}, one gets
\begin{equation}
\label{xex6}
\begin{split}
\lim_{n\to\infty} \frac{\phi_{n-1}}{\ln \frac{1}{a^2_{n}}} 
&= \lim_{n\to\infty} \frac{\phi_{n} - \phi_{n-1}}{\ln \frac{1}{a^2_{n+1}} - \ln \frac{1}{a^2_n}}
= \lim_{n\to\infty}\frac{a_nh}{4\ln(1+\frac{a_n-a_{n+1}}{a_{n+1}})}\\
&=\lim_{n\to\infty} \frac{h}{4\frac{a_n-a_{n+1}}{a_{n+1}a_n}} = \infty.
\end{split}
\end{equation}
This implies that \eqref{xex5} holds for all $n\ge \tilde{N}$ provided that $\tilde{N}>0$ is sufficiently large.
It follows from inequality \eqref{xex5} that 
\begin{equation}
\label{xex7}
\lim_{n\to\infty} a_ne^{\phi_{n-1}} \ge \lim_{n\to\infty} a_n e^{\ln \frac{1}{a^2_{n}}} = \lim_{n\to\infty}\frac{a_n}{a_{n}^2} = \infty.
\end{equation}

{\it Let us prove \eqref{eql9}.} 

Since 
$a_ne^{\phi_{n-1}}\to \infty$ as $n\to\infty$, by \eqref{eql6},
relation \eqref{eql9} holds if the numerator 
$\sum_{i=0}^{n} e^{\phi_{i}}(a_i - a_{i+1})$ in \eqref{eql9} is bounded. 
Otherwise, a discrete analog of L'Hospital's rule yields:
\begin{equation}
\label{eql14}
\begin{split}
M&=\lim_{n\to\infty} \frac{e^{\phi_{n}}(a_n-a_{n+1})}
{e^{\phi_{n}}a_{n+1} - e^{\phi_{n-1}}a_{n}} 
=\lim_{n\to\infty} \frac{a_n-a_{n+1}}
{a_{n+1} - a_{n}e^{-ha_{n}}}\\
&\le \lim_{n\to\infty} \frac{1}{\frac{h a_na_{n}}{(a_n-a_{n+1})}(1-\frac{ha_{n}}{2}) -1}=0.
\end{split}
\end{equation}
Here, we have used \eqref{eqvv}, \eqref{eqqe}, relation $\lim_{n\to\infty}a_n=0$, and the following inequality:
$$
e^{-ha_{n}} \le 1 - ha_{n} + \frac{(ha_{n})^2}{2},\qquad \forall n\ge 0.
$$ 

Lemma~\ref{lemma11} is proved. 
\end{proof}

\section{Main results}
\label{mainsec}

\subsection{An iterative scheme}

Let
$0<a(t)\in C^1(\mathbb{R}_+)$  
satisfy the following conditions: (see also \eqref{eqsxa})
\begin{equation}
\label{eqsxa1}
0< a(t)\searrow 0,
\qquad \nu(t):=\frac{|\dot{a}(t)|}{a^2(t)}\searrow 0,\qquad t\ge 0.
\end{equation}
Let $0<h=const\le 1$ and $a_n:=a(nh),n\ge 0$. Consider the following iterative scheme 
\begin{equation}
\label{3eq12}
u_{n+1} =u_n - \gamma_n[F(u_n)+a_nu_n-f_\delta],\qquad u_0=\tilde{u}_0,
\end{equation}
where $\tilde{u}_0\in H$ and
\begin{equation}
\label{eqk1}
0< h\le \gamma_n \le \frac{2}{\sigma^{-1}_R + 2a_n}, \qquad n\ge 0,
\end{equation}
where $\sigma_R$ is the constant in \eqref{ceq2} and $0<R=const$. 

\begin{theorem}
\label{mainthm}
Let $a(t)$ 
satisfy \eqref{eqsxa1}. 
Assume that $F:H\to H$ is a locally $\sigma$-inverse monotone operator.
Assume that equation $F(u)=f$ has a solution, possibly nonunique. 
Let $f_\delta$ be such that $\|f_\delta-f\|\le \delta$ and  
 $u_0$ be an element of $H$ 
satisfying the inequality:
\begin{equation}
\label{eqj1}
\|F(u_0)-f_\delta\|>C\delta^\zeta,
\end{equation}
where $C>0$ and $\zeta\in (0, 1]$ are constants satisfying 
$C\delta^\zeta>\delta$. 
Let $0<R$ be sufficiently large and $0<h$ and $0<\gamma_n$ satisfy 
\eqref{eqk1}. Let $u_n$ be defined by the iterative process 
\eqref{3eq12}. 
Then there exists a unique $n_\delta$ such that 
\begin{equation}
\label{2eq3}
\|F(u_{n_\delta})-f_\delta\|\le C\delta^\zeta,
\quad \|F(u_n)-f_\delta\|>C\delta^\zeta,\qquad 0\le n < n_\delta,
\end{equation}
where $C$ and $\zeta$ are constants from \eqref{eqj1}. 
\end{theorem}

\begin{remark}
 In \cite{R550} the existence of $n_\delta$ was proved for the choice 
$a_n = d/(c+n)^b,$
 where $b\in (0,1/2)$ and $d>0$ is sufficiently large. However,
it was not quantified in \cite{R550} how large $d$ should be. In this paper 
the existence of
$n_\delta$ is proved for $a_n = d/(c+nh)^b$, for any $d>0,\, c>1,
\, b\in(0,1)$, and $0<h\le \gamma_n$.  This
 guarantees the existence of $n_\delta$ for small $a(0)$ or $d$.
Moreover, our
 condition on $b$ allows $a_n$ to decay faster
than the corresponding sequence $a_n$ in \cite{R550} decays.
Having smaller $a(0)$ and larger $b$ reduces the
cost of computations.

Inequality \eqref{eqj1} is a very natural assumption. Indeed, if it does not hold 
and $\|u_0\|$ is not "large", then $u_0$ can be already considered as an approximate solution to \eqref{aeq1}. 

In general, 
if $R$ in \eqref{eqk1} is large, then the stepsize $h$ in the iterative scheme \eqref{3eq12} is small. 
Consequently, the computation time will be large since the rate of decay of $(a_n)_{n=1}^\infty$ 
is slow.
However, if $F$ is $\sigma$-inverse monotone, i.e., $\sigma_R$ is independent of $R$, then 
it is easy to choose $h$ and $\gamma_n$ to satisfy \eqref{eqk1}.
\end{remark}

\begin{proof}[Proof of Theorem~\ref{mainthm}] 

Let us prove first that there exists $R>0$ such that the sequence $(u_n)_{n=1}^{n_\delta}$ remains inside 
the ball $B(0,R)$. 

We assume without loss of generality that $\delta\in (0,1)$. 
It follows from \eqref{eqstar} and the triangle inequality that
\begin{equation}
\label{inveqzxc5}
a_n\|V_n\| \le \|F(0)-f_\delta\| \le \|F(0)-f\| + \|f_\delta - f\| \le \Gamma,\qquad \forall n\ge 0,\quad \forall \delta \in (0,1),
\end{equation}
where
$$
\Gamma:= \|F(0)-f\| + 1.
$$
From \eqref{inveqzxc5} 
one obtains
\begin{equation}
\label{eqhelloxv}
\|V_n\|\le  \frac{\Gamma}{a_n},\qquad \forall n\ge 0.
\end{equation}

Let $\phi(t)$ be defined as follows (see also \eqref{eqkkk})
\begin{equation}
\label{eqiv}
\phi(t) = \int_0^t a(s)ds.
\end{equation}
From the last inequality in \eqref{eqsxaz} one gets
\begin{equation}
\label{eqpia1}
\lim_{t\to\infty}\phi(t) = \infty.
\end{equation}
We claim that
\begin{equation}
\label{eqff3}
\lim_{t\to\infty} \frac{\int_0^t e^{\phi(s)}\frac{|\dot{a}(s)|}{a(s)}ds}{e^{\phi(t)}} = 0.
\end{equation}
Indeed, if the denominator $\int_0^t e^{\phi(s)}\frac{|\dot{a}(s)|}{a(s)}ds$ is bounded, then 
\eqref{eqff3} is valid because $\lim_{t\to\infty}\phi(t) = \infty$. 
Otherwise L'Hospital's rule and \eqref{eqsxa1} yield
\begin{equation}
\label{eqff3xc}
\lim_{t\to\infty} \frac{\int_0^t e^{\phi(s)}\frac{|\dot{a}(s)|}{a(s)}ds}{e^{\phi(t)}} = 
\lim_{t\to\infty} \frac{e^{\phi(t)}|\dot{a}(t)|}{e^{\phi(t)}a^2(t)} = 0.
\end{equation}

Let
\begin{equation}
\label{eqjkl4}
\qquad K = 1 + \sup_{t\ge 0} e^{a(0)}e^{-\phi(t)} \int_0^t e^{\phi(s)} \frac{|\dot{a}(s)|}{a(s)}ds. 
\end{equation}
It follows from \eqref{eqff3} that $K$ is bounded. 

Let 
$V_\delta(t)$ solves the equation
\begin{equation}
\label{eq}
F(V_\delta(t)) + a(t)V_\delta(t) - f_\delta = 0.
\end{equation}
It follows from Lemma~\ref{lemma5} that
\begin{equation}
\label{eqjvcc}
\|F(V_\delta(t)) - f_\delta\| \le \|F(V_n) - f_\delta\|,\qquad  
\forall t\ge nh.
\end{equation}
Relation \eqref{eqff3}, Lemma~\ref{lemma3} and \eqref{eqjvcc} imply that there exists $T>0$ such that 
the following inequality holds $\forall t\in [T,T+1]$:
\begin{equation}
\label{eqff1}
\begin{split}
\|F(V_\delta(t)) &- f_\delta\| + e^{-\phi(t)}\psi_0 \\
&+  e^{a(0)} e^{-\phi(t)}\int_0^t e^{\phi(s)}|\dot{a}(s)| \bigg{(}\frac{\Gamma K}{a(s)}+ w_0\bigg{)} ds < C\delta^\zeta,
\end{split}
\end{equation}
where 
\begin{equation}
\label{eqjpeg}
\psi_0 = \|F(u_0)+a_0 u_0 -f_\delta \|,\qquad w_0 = \|u_0 - V_0\|. 
\end{equation}
Let
\begin{equation}
\label{eqff6}
R := \|V_\delta(T)\|K + w_0.
\end{equation}

Let 
$$
0<h\le \min(1,\frac{2}{\sigma_R^{-1}+2a(0)}).
$$ 
Let $N$ be the largest integer such that $Nh\le T$. 
Let us prove by induction that 
the sequence $(u_n)_{n=1}^{N}$ stays in side the ball $B(0,R)$. 
To prove this it suffices to prove that
\begin{equation}
\label{eqpia2}
\|u_n - V_n\| \le w_0+\|V_n\|(K-1),\qquad n=0,1,...,N.
\end{equation}
Inequality \eqref{eqpia2} holds for $n=0$, by \eqref{eqjpeg}. 
Assume that \eqref{eqpia2} holds for $0\le n<N$. Let us prove that \eqref{eqpia2} also holds for $n+1$. 

It follows from equation \eqref{3eq12} that
\begin{equation}
\label{eqxvzc}
\begin{split}
u_{n+1} - V_n &= u_n - V_n - \gamma_n\bigg{[}F(u_n) + a_n u_n - F(V_n) - a_nV_n\bigg{]}\\
&= (1-\gamma_n a_n)\bigg{[}u_n - V_n -\frac{\gamma_n}{1-\gamma_na_n}\big{(}F(u_n) - F(V_n)\big{)}\bigg{]}.
\end{split}
\end{equation}
This and Lemma~\ref{lemma2} imply
\begin{equation}
\label{eq**}
\|u_{n+1}-V_n\|\le \|u_n - V_n\|(1 - \gamma_n a_n). 
\end{equation}
From \eqref{eq**}, the triangle inequality and \eqref{eqjkl7} one gets 
\begin{equation}
\label{eqjvc}
\begin{split}
\|u_{n+1}-V_{n+1}\|&\le \|u_n - V_n\|(1 - \gamma_n a_n) + \|V_{n+1} - V_n\|\\
&\le \|u_n - V_n\| e^{-ha_n} + \frac{a_n - a_{n+1}}{a_n}\|V_{n+1}\|.
\end{split}
\end{equation}
Here we have used the inequality: $1 - \gamma_n a_n\le  e^{-ha_n}$ where $0<h\le \gamma_n$ and $n\ge 0$. 
From \eqref{eqjvc} one gets by induction the following inequality:
\begin{equation}
\label{eq***}
\|u_{n+1}-V_{n+1}\| 
\le w_0e^{-\phi_n} + e^{-\phi_n}\sum_{i=1}^n e^{\phi_i}\frac{a_i - a_{i+1}}{a_i}\|V_{i+1}\|,
\end{equation}
where
$$
\phi_{n} = \sum_{i=0}^n ha_i,\qquad n\ge 0.
$$
From \eqref{eq***}, Lemma~\ref{lemma5}, Lemma~\ref{lemma10}, and \eqref{eqjkl4}, one obtains
\begin{equation}
\label{eqjlk2}
\begin{split}
\|u_{n+1}-V_{n+1}\| &\le w_0 + \|V_{n+1}\|e^{-\phi_n}\sum_{i=1}^n e^{\phi_i}\frac{a_i - a_{i+1}}{a_i}\\
&\le w_0 + \|V_{n+1}\|e^{a(0)}e^{-\phi((n+1)h)}\int_0^{(n+1)h} e^{\phi(s)}\frac{|\dot{a}(s)|}{a(s)}ds\\
&\le w_0 + \|V_{n+1}\|(K-1),
\end{split}
\end{equation}
where $\phi(t)$ is defined by \eqref{eqiv}.

Hence, \eqref{eqpia2} holds for $n+1$. 
Thus, by induction \eqref{eqpia2} holds for $0\le n\le N$. 
Inequalities \eqref{eqpia2}, \eqref{eqff6}, and the inequality $\|V_n\|\le \|V_\delta(T)\|$, $\forall n\le N$ (see Lemma~\ref{lemma5}),  
 imply that the sequence $(u_n)_{n=1}^{N}$ remains inside the 
ball $B(0,R)$

{\it Let us prove the existence of $n_\delta$.} 

Denote $g_n:=g_{n,\delta}:=F(u_n) +a_nu_n -f_\delta$. 
Equation \eqref{3eq12} can be rewritten as  
\begin{equation}
\label{eqpp3}
u_{n+1}-u_n = -\gamma_n g_n,\qquad n\ge 0.
\end{equation}
This implies
\begin{equation}
\label{beq32}
\begin{split}
g_{n+1} =& g_n + a_{n+1}(u_{n+1}-u_n)+ F(u_{n+1})-F(u_n)+(a_{n+1}-a_n)u_n\\
=& -\frac{u_{n+1}-u_n}{\gamma_n} + a_{n+1}(u_{n+1}-u_n)+ F(u_{n+1})-F(u_n)\\
& +(a_{n+1}-a_n)u_n\\
=& -\frac{1-\gamma_n a_{n+1}}{\gamma_n}\bigg{[}(u_{n+1}-u_n) - \frac{\gamma_n}{1-\gamma_n a_{n+1}}\big{(}F(u_{n+1})-F(u_n)\big{)}\bigg{]}\\
& +(a_{n+1}-a_n)u_n.
\end{split}
\end{equation} 
Denote $\psi_n=\|g_n\|$. 
It follows from \eqref{beq32} that
\begin{equation}
\label{2eq5}
\begin{split}
\psi_{n+1} \le & \frac{1-\gamma_n a_{n+1}}{\gamma_n}\bigg{\|}(u_{n+1}-u_n) - \frac{\gamma_n}{1-\gamma_n a_{n+1}}\big{(}F(u_{n+1})-F(u_n)\big{)}\bigg{\|}\\
 &+ (a_n -a_{n+1})\|u_n\|.
\end{split}
\end{equation}
From Lemma~\ref{lemma2} and \eqref{eqpp3} we get the following inequality:
\begin{equation}
\label{eq02}
\begin{split}
\bigg{\|}(u_{n+1}-u_n) - \frac{\gamma_n}{1-\gamma_n a_{n+1}} &\big{(}F(u_{n+1})-F(u_n)\big{)}\bigg{\|} 
\le \|u_{n+1}-u_{n}\|=\gamma_n \psi_n,
\end{split}
\end{equation}
for all $0\le n\le N-1$. 
From \eqref{eq02} and \eqref{2eq5} one gets
\begin{equation}
\label{eqpp11}
\psi_{n+1} \le (1-\gamma_n a_{n+1})\psi_n + (a_n -a_{n+1})\|u_n\|.
\end{equation}
Note that one has: $1- h a_{n+1} \le e^{-a_{n+1}h}$, $\forall n\ge 0$. This, 
inequalities \eqref{eqpp11} and \eqref{eqpia2} imply
\begin{equation}
\label{beq39}
\begin{split}
\psi_{n+1} 
&\le e^{-a_{n+1}h}\psi_n + (a_n -a_{n+1})\big{(}\|V_n\|K + w_0\big{)},\qquad 0\le n\le N-1.
\end{split}
\end{equation}
From inequality \eqref{beq39} one gets by induction the following inequality:
\begin{equation}
\label{1eq17}
\psi_{n}\le \psi_{0}e^{-\phi_{n-1}} + e^{-\phi_{n-1}}
\sum_{i=0}^{n-1} e^{\phi_{i}}(a_i -a_{i+1})\big{(}\|V_i\|K + w_0\big{)},\qquad 1\le n\le N,
\end{equation}
where $\phi_n$ is defined by \eqref{eqiv}.

Since $F(V_n) + a_nV_n -f_\delta=0$, one gets
\begin{equation}
\label{xex1}
g_n=F (u_n) -F(V_n)+a_n(u_n-V_n).
\end{equation}
This and \eqref{ceq2.1} imply
\begin{equation}
\label{beq35}
a_n\|u_n - V_n\|^2 \le \langle g_n, u_n-V_n \rangle \le 
\|u_n - V_n\|\psi_n,
\end{equation}
and
\begin{equation}
\label{beq36}
\|F(u_n)-F(V_n)\|^2\le \langle g_n, F(u_n)-F(V_n) \rangle
\le \psi_n\|F(u_n)-F(V_n)\|.
\end{equation}
Inequalities \eqref{beq35} and \eqref{beq36} imply 
\begin{equation}
\label{1eq16}
a_n\|u_n-V_n\|\le \psi_n,\quad \|F(u_n)-F(V_n)\|\le \psi_n.
\end{equation}


From \eqref{1eq17}, and 
\eqref{1eq16}, one gets, for $0\le n\le N$, the following inequality:
\begin{equation}
\label{eqj5}
\|F(u_n)-F(V_n)\| \le 
\psi_{0}e^{-\phi_{n-1}} + e^{-\phi_{n-1}}\sum_{i=0}^{n-1} 
e^{\phi_{i}}(a_i -a_{i+1})\big{(}\|V_i\|K + w_0\big{)}.
\end{equation}
This, the triangle inequality,  
%
%
and inequalities \eqref{eqhelloxv} imply
\begin{equation}
\label{inveqr9}
\begin{split}
\|F(u_n)-f_\delta\| \le &\|F(V_n)-f_\delta\|+ \psi_{0}e^{-\phi_{n-1}}\\
&+ e^{-\phi_{n-1}}\sum_{i=0}^{n-1}e^{\phi_{i}}(a_i - a_{i+1})\bigg{(}\frac{\Gamma K}{a_i} + w_0\bigg{)},
\qquad  0\le n\le N.
\end{split}
\end{equation}
From \eqref{eqiv} and the fact that $a(t)$ is decreasing one gets
\begin{equation}
\label{eqvnhc}
\phi_{n-1} = \sum_{i=0}^{n-1} ha(ih)\ge \int_0^{nh} a(s)ds=\phi(nh),\qquad n\ge 1.
\end{equation}

Inequality \eqref{inveqr9} with $n=N$, equation \eqref{eqff6}, the inequality $T-1< Nh \le T$, 
by the definition of $N$, and Lemma~\ref{lemma10} imply
\begin{equation}
\label{xhoqnei}
\begin{split}
\|F(u_N)-f_\delta\| \le &\|F(V_\delta(Nh))-f_\delta\|+ \psi_{0}e^{-\phi(Nh)}\\
&+ e^{a(0)}e^{-\phi(Nh)}\int_{0}^{Nh}e^{\phi(s)}|\dot{a}(s)|\bigg{(}\frac{\Gamma K}{a(s)} + w_0\bigg{)}ds  < C\delta.
\end{split}
\end{equation}
This implies the existence of $n_\delta$.

The uniqueness of $n_\delta$, satisfying \eqref{2eq3}, follows 
from its definition \eqref{2eq3}. 

Theorem~\ref{mainthm} is proved. 
\end{proof}

\begin{theorem}
\label{theorem2}
Let $F,f,f_\delta$ and $u_\delta$ be as in Theorem~\ref{mainthm} and 
and $y$ be the minimal-norm solution to the equation $F(u)=f$. 
Let $0<(\delta_m)_{m=1}^\infty$ be a sequence such that 
$\delta_m\to 0$. If the sequence $\{n_{\delta_m}\}_{m=1}^\infty$ 
is bounded, and $\{n_{m_j}\}_{j=1}^\infty$ is a convergent subsequence,
then
\begin{equation}
\label{eqextr5}
\lim_{j\to\infty} u_{n_{m_j}} = u^\star,
\end{equation} 
where $u^\star$ is a solution to the equation $F(u)=f$. 
If 
\begin{equation}
\label{eqextra6}
\lim_{m\to\infty}n_m = \infty,
\end{equation}
and $\zeta\in(0,1)$, then
\begin{equation}
\label{eqextra7}
\lim_{m\to\infty} \|u_{n_m} - y\| = 0.
\end{equation}
\end{theorem}

\begin{proof}
Let us first prove \eqref{eqextra7} assuming that \eqref{eqextra6} holds. 
For simplicity we will prove that
\begin{equation}
\label{extra7}
\lim_{\delta\to 0} \|u_\delta - y\| = 0,
\end{equation}
under the assumption that
\begin{equation}
\label{extra6}
\lim_{\delta\to 0}n_\delta = \infty,
\end{equation}

From \eqref{extra6}, Lemma~\ref{lemma11}, and 
the fact that the sequence $(\|V_n\|)_{n=0}^\infty$ is increasing, one gets the following inequalities for sufficiently small $\delta>0$:
\begin{equation}
\label{eqextra8}
\psi_0 e^{-\phi_{n_\delta-2}}\le 
\frac{1}{4}a_{n_\delta-1}\|V_0\|
<\frac{1}{4}a_{n_\delta-1}\|V_{n_\delta-1}\|,
\end{equation}
and 
\begin{equation}
\label{eqextra9}
e^{-\phi_{n_\delta - 2}}\sum_{i=0}^{n_\delta-2} 
e^{\phi_{i}}(a_i -a_{i+1})(\|V_i\|+K) \le \frac{1}{2}a_{n_\delta-1}\|V_{n_\delta-1}\|. 
\end{equation}

From \eqref{2eq3}, and \eqref{inveqr9} with $n=n_\delta-1$, 
\eqref{eqextra8} and \eqref{eqextra9}, one obtains
\begin{equation}
\label{eql10}
C \delta^\zeta < a_{n_\delta-1}\|V_{n_\delta-1}\| (1+\frac{1}{2}+\frac{1}{4})
\le \frac{7}{4}\bigg{(}a_{n_\delta-1}\|y\| + \delta\bigg{)},
\end{equation}
for all $0<\delta$ sufficiently small. 
This and the relation $\lim_{\delta\to0}\frac{\delta^\zeta}{\delta} =\infty$
for a fixed $\zeta\in(0,1)$ imply
\begin{equation}
\label{eql11}
\limsup_{\delta\to 0}\frac{\delta^\zeta}{a_{n_\delta}} \le \frac{7\|y\|}{4C}<\frac{2\|y\|}{C}.
\end{equation}
Inequalities \eqref{eql11}, $\delta<C\delta^\zeta$, and \eqref{2eq1} imply, for sufficiently small $\delta>0$,
the following inequality
\begin{equation}
\label{eql12}
\|V_n\| \le \|y\|+\frac{\delta}{a_{n_\delta}} < \tilde{C}:=\|y\|+ 2\|y\|,\qquad 
0\le n\le n_\delta. 
\end{equation}
Using estimate \eqref{eql12}, one obtains:
\begin{equation}
\label{eql13}
\lim_{\delta\to 0} \frac{\sum_0^{n_\delta-1} e^{\phi_{i}}(a_i - a_{i+1})\|V_i\|}
{e^{\phi_{n_\delta-1}}a_{n_\delta}} \le
\tilde{C}
\lim_{\delta\to 0}
\frac{\sum_0^{n_\delta-1} e^{\phi_{i}}(a_i - a_{i+1})}
{e^{\phi_{n_\delta-1}}a_{n_\delta}}.
\end{equation}

It follows from \eqref{eql9} and \eqref{eql13} that
\begin{equation}
\label{eql15}
\lim_{\delta\to 0} \frac{\sum_{i=0}^{n_\delta-1} e^{\phi_{i}}(a_i - a_{i+1})\|V_i\|}
{e^{\phi_{n_\delta-1}}a_{n_\delta}} = 0.
\end{equation}

From \eqref{1eq17} and \eqref{1eq16} one gets
\begin{equation}
\label{eq123}
\|u_n - V_n\| \le \frac{e^{-\phi_{n-1}}\psi_0}{a_n} + 
\frac{e^{-\phi_{n-1}}}{a_n} \sum_{i=0}^{n-1} e^{-\phi_{i}}(a_i - a_{i+1})\big{(}\|V_i\|K + w_0\big{)}.
\end{equation}
This, \eqref{extra6}, and \eqref{eql15} one obtains:
\begin{equation}
\label{eqj6}
\lim_{\delta\to 0} \|u_{n_\delta} - V_{n_\delta}\| = 0.
\end{equation}
It follows from \eqref{eql11} that
\begin{equation}
\label{eqo11}
\lim_{\delta\to 0} \frac{\delta}{a_{n_\delta}} = 0.
\end{equation}

From the triangle inequality and inequality 
\eqref{rejected11} one obtains:
\begin{equation}
\label{eqhic56}
\begin{split}
\|u_{n_\delta} - y\| &\le \|u_{n_\delta} - V_{n_\delta}\| + 
\|V_{n_\delta} - V_{0,n_\delta}\| + \|V_{0,n_\delta} - y\|\\
&\le \|u_{n_\delta} - V_{n_\delta}\| + \frac{\delta}{a_{n_\delta}} + \|V_{0,n_\delta}-y\|.
\end{split}
\end{equation}
Note that $V_{0,n_\delta} = \tilde{V}_{a(n_\delta),0}$ (cf. \eqref{2eq2}). 
From \eqref{eqj6}--\eqref{eqhic56}, \eqref{extra6}, and Lemma~\ref{rejectedlem}, one obtains
\eqref{extra6}. 

{\it Let us prove \eqref{eqextr5}.}

If $n>0$ is fixed, then $u_n$ is a continuous function of $f_\delta$. 
Denote 
\begin{equation}
\label{eqextra10}
u^\star := u_N^\star:=\lim_{j\to\infty} u_{n_{\delta_{m_j}}},
\end{equation}
where
\begin{equation}
\label{eqextra11}
\lim_{j\to\infty} n_{m_j} = N.
\end{equation}
From \eqref{eqextra10} and the continuity of $F$, one obtains:
\begin{equation}
\label{eqextra12}
\|F(u^\star) - f\| = \lim_{j\to\infty}\|F(u_{n_{\delta_{m_j}}}) - f_{\delta_{m_j}}\|
\le \lim_{j\to\infty}C\delta_{m_j}^\zeta = 0.
\end{equation}
Thus, $u^\star$ is a solution to the equation $F(u)=f$, and \eqref{eqextr5} is proved. 

Theorem~\ref{theorem2} is proved.
\end{proof}

Let us assume in addition that $a(t)$ satisfies the following 
inequalities
\begin{equation}
\label{eqxtr4}
2\ge a(0),\qquad \nu(0)=\frac{|\dot{a}(0)|}{a^2(0)}\le \frac{1}{10}. 
\end{equation}

\begin{remark}
\label{lilo}
Let $b\in(0,1)$, $c\ge5$, $d>0$ and 
$$
a(t)=\frac{d}{(c+t)^b},
\qquad \frac{10b}{c^{1-b}}\le d\le 2c^b.
$$
Then $a(t)$ satisfies \eqref{eqsxa} and \eqref{eqxtr4}.
\end{remark}

We have the following result

\begin{theorem}
\label{theorem3}
Let $a(t)$ satisfy \eqref{eqsxa} and \eqref{eqxtr4}.
Let $F,f,f_\delta$ and $u_\delta$ be as in Theorem~\ref{mainthm}. Assume that $u_0$ satisfies either
\begin{equation}
\label{eqextr2}
\psi_0=\|F(u_0)+a_0u_0 - f_\delta\| \le \theta \delta^\zeta,\qquad 0<\theta<C,
\end{equation}
or
\begin{equation}
\label{eqextr1}
\|F(u_0) + a_0u_0 -f_\delta\| \le \frac{1}{8}a_0\|V_0\|.
\end{equation}
Assume $\zeta\in(0,1)$. 
Then
\begin{equation}
\label{eqextr3}
\lim_{\delta\to 0}n_\delta = \infty. 
\end{equation}
\end{theorem}

\begin{remark}
\label{remark3.8}
{\rm
The element $u_0$ satisfying \eqref{eqextr2} can be obtained easily by the following 
fixed point iterations:
\begin{equation}
\label{eqkkk3}
v_{n+1} = v_n - \gamma(F(v_n) + a_0 v_n - f_\delta),\qquad n\ge 0,
\end{equation}
where $v_0\in B(0,R)$, $0<R$ is sufficiently large, and $\gamma$ is chosen so that
\begin{equation}
\label{eqkkk4}
0 < \gamma < \frac{2}{\sigma_R^{-1} + 2a_0}.  
\end{equation}
Note that the operator $G(v):=v - \gamma(F(v) + a_0 v - f_\delta)$ is a contraction 
map by Lemma~\ref{lemma2}.

Inequality \eqref{eqextr1} is a sufficient condition for 
the following inequality to hold (see also \eqref{1eq20} below)
\begin{equation}
\label{eqj2}
e^{-\varphi_n}\psi_0\le \frac{1}{8}a_n\|V_n\|,\qquad t\ge 0. 
\end{equation}
By similar arguments as in the proof of Lemma~\ref{lemma11} one can prove that 
$$
\lim_{n\to \infty}e^{\varphi_{n}}a_{n}= \infty.
$$
In the proof of Theorem~\ref{theorem3} inequality \eqref{eqj2} (or \eqref{1eq20})
is used at $n=n_\delta$.
The stopping time $n_\delta$ is often sufficiently large for 
the quantity $e^{\varphi_{n_\delta}}a_{n_\delta}$ to be large. 
In this case  
inequality \eqref{1eq20} with $n=n_\delta$ is satisfied for a wide range of 
$u_0$. 

It is an {\it open problem} to choose $\zeta$ (see \eqref{eqj1}) which 
is optimal in some sense. In practice it is {\it natural} to choose $C$ 
and $\zeta$ so that $C\delta^\zeta$ is close to $\delta$. It is because  
if $v$ is a solution to the equation $F(u)=f$, then $\|F(v)-f_\delta\| = \|f-f_\delta\|\le \delta$.
}
\end{remark}

Let us now prove Theorem~\ref{theorem3}.

\begin{proof}
[Proof of Theorem~\ref{theorem3}]

Let us prove \eqref{eqextr3} assuming that \eqref{eqextr1} holds. 
When \eqref{eqextr2} holds, instead of \eqref{eqextr1}, the proof follows similarly. 

It follows from \eqref{eqpp11}, the triangle inequality, and \eqref{1eq16} that
\begin{equation}
\label{eqpia4}
\begin{split}
\psi_{n+1} &\le (1-ha_{n+1})\psi_n + (a_n-a_{n+1})\|u_n - V_n\| + (a_n-a_{n+1})\|V_n\|\\ 
&\le (1-ha_{n+1})\psi_n + \frac{a_n-a_{n+1}}{a_n} \psi_n + (a_n-a_{n+1})\|V_n\|,
\end{split}
\end{equation}
for $0\le n\le n_\delta -1$. 
From \eqref{eqz8} and \eqref{eqxtr4} one gets
\begin{equation}
\label{eqpia5}
1 - ha_{n+1} + \frac{a_n - a_{n+1}}{a_n} \le 1 - \frac{h a_{n+1}}{2}\le e^{-\frac{h a_{n+1}}{2}}. 
\end{equation}
From \eqref{eqpia4} and \eqref{eqpia5} one obtains
\begin{equation}
\label{eqpia6}
\psi_{n+1} \le e^{-\frac{h a_{n+1}}{2}}\psi_n + (a_n-a_{n+1})\|V_n\|.
\end{equation}
This implies
\begin{equation}
\label{eqpia7}
\psi_{n} \le e^{-\varphi_n}\psi_0 + e^{-\varphi_n}\sum_{i=0}^{n-1}e^{-\varphi_i}(a_i-a_{i+1})\|V_i\|,\qquad 
1\le n\le n_\delta.
\end{equation}
where
\begin{equation}
\label{eqklj1}
\varphi_n = \sum_{i=1}^n \frac{h a_n}{2}.
\end{equation}

It follows from the triangle inequality, \eqref{eq*}, \eqref{1eq16}, and \eqref{eqpia7} that
\begin{equation}
\label{1eq18}
\begin{split}
\|F(u_n)-f_\delta\|&\ge \|F(V_n)-f_\delta\|-\|F(V_n)-F(u_n)\|\\
&\ge a_n\|V_n\| - \psi_0e^{-\varphi_n} - 
e^{-\varphi_n}\sum_{i=0}^{n-1} e^{\varphi_{i+1}} (a_i -a_{i+1})  \|V_i\|.
\end{split}
\end{equation} 
From Lemma~\ref{lemauxi2} one obtains
\begin{equation}
\label{1eq19}
\frac{1}{2}a_n\|V_n\| \ge  e^{-\varphi_n}\sum_{i=0}^{n-1}e^{\varphi_{i+1}}(a_i -a_{i+1}) \|V_i\|.
\end{equation}

From the relation $\psi_n=\|F(u_n)+a_nu_n -f_\delta\|$ (cf. \eqref{xex1}) and \eqref{eqextr1} one gets
\begin{equation}
\label{23eq6}
\psi_0e^{-\varphi_n}\le \frac{1}{8}a_0\|V_0\|e^{-\varphi_n},\qquad 
n\ge 0.
\end{equation}
It follows from \eqref{eqsxa} that 
\begin{equation}
\label{26eq5}
a_1\le a_{n+1}e^{\varphi_n},\qquad \forall n\ge 0.
\end{equation}
Indeed, inequality $a_1\le a_{n+1}e^{\varphi_n}$ is obviously true for 
$n=0$, and $a_{n+1}e^{\varphi_n}$ is an increasing sequence because
\begin{equation}
\label{eqp1}
\begin{split}
a_{n+1}e^{\varphi_{n}} - a_{n}e^{\varphi_{n-1}}
&= e^{\varphi_{n-1}}(a_{n+1}e^{\frac{h a_n}{2}} - a_n)\\
&\ge e^{\varphi_{n-1}}(a_{n+1}+ \frac{h a_n}{2}a_{n+1} - a_n)\\
&= e^{\varphi_{n-1}} a_n a_{n+1}(\frac{h}{2} - \frac{a_n -a_{n+1}}{a_n a_{n+1}})\ge 0,
\end{split}
\end{equation}
by \eqref{eqz8} and \eqref{eqxtr4}. 
From \eqref{26eq5}, \eqref{eqxtr4} and \eqref{eql2} one gets
\begin{equation}
\label{eqpp6}
e^{-\varphi_n}a_0\le a_{n+1}\frac{a_0}{a_1} < 2 a_{n+1},\qquad n\ge 0.
\end{equation}

Inequalities \eqref{23eq6} and \eqref{eqpp6} imply
\begin{equation}
\label{1eq20}
e^{-\varphi_n}\psi_0 
\le \frac{1}{4} a_{n+1}\|V_{0}\|
\le \frac{1}{4} a_{n}\|V_{n}\|,\quad n\ge 0,
\end{equation}
where we have used the inequality $\|V_{n'}\|\le \|V_n\|$ for $n'\le n$, 
established in Lemma~\ref{lemma5}.
From \eqref{1eq18}, \eqref{1eq19} and \eqref{1eq20}, one gets
$$
\|F(u_{n_\delta})-f_\delta\|\ge a_{n_\delta}\|V_{n_\delta}\| - \frac{1}{4}a_{n_\delta}\|V_{n_\delta}\| - \frac{1}{2}a_{n_\delta}\|V_{n_\delta}\| = \frac{1}{4}a_{n_\delta}\|V_{n_\delta}\|.
$$
This and \eqref{2eq3} imply
$$
C\delta^\zeta \ge \|F(u_{n_\delta})-f_\delta\|\ge \frac{1}{4}a_{n_\delta}\|V_{n_\delta}\|.
$$
Thus,
\begin{equation}
\label{xex3}
\lim_{\delta\to0}a_{n_\delta}\|V_{n_\delta}\|\le 
\lim_{\delta\to0}4C\delta^\zeta = 0.
\end{equation}
From \eqref{rejected11} and the triangle inequality we obtain
\begin{equation}
\label{xex2}
a_{n_\delta}\|V_{0,n_\delta}\| \le a_{n_\delta}\|V_{n_\delta}\| + a_{n_\delta}\|V_{n_\delta}-V_{0,n_\delta}\|
\le  a_{n_\delta}\|V_{n_\delta}\| +\delta.
\end{equation}
This and \eqref{xex3} imply
\begin{equation}
\label{exe4}
\lim_{\delta\to0}a_{n_\delta}\|V_{0,n_\delta}\| = 0.
\end{equation}
Since $\|V_{0,n_\delta}\ge\|V_{0,0}\|>0$, relation \eqref{exe4} implies 
$\lim_{\delta\to0}a_{n_\delta}=0$. 
Since $0<a_n\searrow 0$, it follows that  
\eqref{eqextr3} holds. 

Theorem~\ref{theorem3} is proved. 
\end{proof}

Instead of using iterative scheme \eqref{3eq12} one may use the following iterative scheme
\begin{equation}
\label{3eq12.3}
u_{n+1} =u_n - \gamma_n[F(u_n)+a_n(u_n - \bar{u})-f_\delta],\qquad u_0=\tilde{u}_0,
\end{equation}
where $\bar{u}, \tilde{u}_0\in H$. Denote $\tilde{F}(u):=F(u + \bar{u})$.
If $F$ is a locally $\sigma$-inverse monotone operator then so is $\tilde{F}$.
Using Theorem~\ref{mainthm} with $F:=\tilde{F}$, one gets the following corollary:

\begin{corollary}
\label{cor1}
Let $a(t)$ satisfy \eqref{eqsxa} and \eqref{eqxtr4}. 
Let $0<R =const$ be sufficiently large and $h$ and $\gamma_n$ satisfy \eqref{eqk1}. 
Assume that $F:H\to H$ is a locally $\sigma$-inverse monotone operator,
and $u_0$ is an element of $H$, 
satisfying inequality 
\begin{equation}
\label{xxeqj1}
\|F(u_0)-f_\delta\|>C\delta^\zeta>\delta,
\end{equation}
where $C>0$ and $0<\zeta\le 1$ are constants. 
Assume also that $u_0$ satisfy either
$$
\|F(u_0) + a_0(u_0 - \bar{u}) -f_\delta\| \le \frac{1}{8}a_0\|V_0\|,
$$
or
$$
\|F(u_0) + a_0(u_0 - \bar{u}) -f_\delta\| \le \theta\delta^\gamma,\qquad 0<\theta=const<C.
$$
 Assume that equation $F(u)=f$ has a solution, possibly 
nonunique,
and $z\in B(u_0,R)$ is the solution with minimal distance to $\bar{u}$.
Let $f_\delta$ be such that $\|f_\delta-f\|\le \delta$. 
Let $u_n$ be defined by \eqref{3eq12.3}. 
Then there exists a unique $n_\delta$ such that 
\begin{equation}
\label{xx2eq3}
\|F(u_{n_\delta})-f_\delta\|\le C\delta^\zeta,
\quad \|F(u_n)-f_\delta\|>C\delta^\zeta,\qquad 0\le n < n_\delta,
\end{equation}
where $C$ and $\zeta$ are constants from \eqref{eqj1}. If $\zeta\in (0,1)$ and 
$n_\delta$ satisfies \eqref{2eq3}, then
\begin{equation}
\label{xx2eq4}
\lim_{\delta\to 0} \|u_{n_\delta} - z\|=0.
\end{equation}
\end{corollary}

\subsection{An algorithm for solving equations with $\sigma$-inverse operators}

Let us formulate an algorithm for solving equations with $\sigma$-inverse operators.

\underline{\textbf{Algorithm 1}}
\begin{enumerate}
\item{Estimate the constant $\sigma=\sigma_R$ in \eqref{ceq2}.}
\item{Choose an $a(t)$ satisfying \eqref{eqsxa}.}
\item{Choose $h = \frac{2}{\sigma^{-1}+2a(0)}$ and $\gamma_n$ to satisfy conditions \eqref{eqk1}.}
\item{Find an initial approximation $u_0$ for $y$ or simply set $u_0 = 0$.}
\item{Compute $u_n$ by formula \eqref{eqqe}, use \eqref{2eq3} to stop the iterations at $n_\delta$ 
and use $u_{n_\delta}$ as an approximate solution to the equation 
$F(u)=f$.}
\end{enumerate}

Theorem~\ref{theorem2} guarantees the convergence of $u_{n_\delta}$, computed by 
{\textbf{Algorithm 1}},  
to, at least, a solution to $F(u)=f$. If the equation $F(u)=f$ has a unique solution, 
then $u_{n_\delta}$ converges to this unique solution. 

If one chooses $a(t)$ to satisfy 
\eqref{eqxtr4} in addition, and $u_0$ to satisfy \eqref{eqextr2} or 
\eqref{eqextr1}, 
then $n_\delta\to \infty$ as $\delta\to 0$ as proved in Theorem~\ref{theorem3}.
Consequently, $u_{n_\delta}$ converges to the minimal-norm solution $y$ 
as stated by Theorem~\ref{theorem2}.

Note that the element $u_0$ satisfying \eqref{eqextr2} can be found from iteration \eqref{eqkkk3}. 
Mover, in practice $n_\delta$ is often large when $\delta$ is sufficiently small. 
Thus, in practice one can also use $u_0=0$ as pointed out in Remark~\ref{remark3.8}. 

{\textbf{Algorithm 1}} can also be implemented for solving equations with locally 
$\sigma$-inverse operators. Since the constant $\sigma_R$ depends on $R$, 
one should choose $R$ sufficiently large so that the sequence $(u_n)_{n=1}^{n_\delta}$ remains in side the ball $B(0,R)$. 
However, if one chooses $R$ too large then $h$ and $\gamma_n$ satisfying  
\eqref{eqk1} are small. 
Consequently, the computation cost will be large. Thus, $R$ should be chosen not too small so that the sequence $(u_n)_{n=1}^{n_\delta}$ remains in side the ball $B(0,R)$ 
and not too large so that the computation cost is not large. 
The choice of $R$ varies from problems to problems. 

\section{Numerical experiments}
  
Let us do a numerical experiment solving nonlinear 
integral equation 
\eqref{aeq1} with 
\begin{equation}
\label{1eq41}
F(u):= B(u)+ \arctan^3(u):=\int_0^1e^{-|x-y|}u(y)dy + \arctan^3(u).
\end{equation}
The operator $B$ is compact in $H=L^2[0,1]$. 
One has
$$
\langle \arctan^3{u}-\arctan^3{v},u-v\rangle = \int_0^1(\arctan^3{u}-\arctan^3{v})(u-v)dx \ge 0,
$$ 
and 
$$
e^{-|x|} = \frac{1}{\pi}\int_{-\infty}^\infty \frac{e^{i\lambda x}}{1+\lambda^2} d\lambda = \sqrt{\frac{2}{\pi}} 
\mathcal{F}^{-1}\bigg{(}\frac{1}{1+\lambda^2} \bigg{)}(x),
$$
where $\mathcal{F}$ denotes the Fourier transform. 
Therefore, $\langle B(u-v),u-v\rangle\ge0$, so 
$$
\langle F(u)-F(v),u-v\rangle\ge0,\qquad \forall u,v\in H.
$$

The Fr\'{e}chet derivative of $F$ is:
\begin{equation}
\label{eq44}
F'(u)h = \frac{3\arctan^2{u}}{1+u^2}h  + 
\int_0^1 e^{-|x-y|}h(y) dy.
\end{equation}
It follows from \eqref{eq44} that $F'$ is selfadjoint and uniformly bounded. Thus, $F$ is a $\sigma$-inverse operator. 
Moreover, one can prove that
$$
\|F'(u)\| \le \sqrt{\frac{2}{\pi}}+\sup_{x\ge 0}\frac{3\arctan^2{x}}{1+x^2}<1 + \sqrt{\frac{2}{\pi}},\qquad \forall u\in H.
$$
This and \eqref{eqsigma} imply that
$$
\sigma_R^{-1} < 1 + \sqrt{\frac{2}{\pi}},\qquad \forall R>0.
$$
Thus, if $a(0)<1-\sqrt{\frac{2}{\pi}}$, then \eqref{eqk1} holds for $\gamma_n=h=1$. 
Therefore, the existence of $n_\delta$ is 
guaranteed with $a_n = \frac{a(0)}{(5+n)^{0.99}}$ and $\gamma_n=1$ by Theorem~\ref{mainthm}. It follows from \eqref{1eq41} that equation $F(u)=f$ has 
not more than one solution for any $f\in H$. Thus if $(\delta_m)_{m=1}^\infty$ is a sequence decaying to $0$ and $n_{\delta_{m_j}}$ is any convergent subsequence of $n_{\delta_m}$, then one gets $u_{n_{\delta_{m_j}}}\to y$, the unique solution to 
$F(u)=f$, by Theorem~\ref{theorem2}. 

If $u(x)$ vanishes on a set of positive Lebesgue's measure, then $F'(u)$
is not boundedly invertible.
If $u\in C[0,1]$ vanishes even at one point $x_0$, then $F'(u)$ is not boundedly invertible in $H$. In this case equation $F(u)=f$ cannot be solved by classical 
methods such as Newton's method or Gauss-Newton method. 

Let us use the iterative process \eqref{3eq12}:
\begin{equation}
\begin{split}
u_{n+1} &= u_n - \gamma_n[F(u_n)+a_nu_n - f_\delta],\\
u_0 &= 0.
\end{split}
\end{equation}
We stop iterations at $n:=n_\delta$ such that the following inequality holds
\begin{equation}
\label{eq53}
\|F(u_{n_\delta}) - f_\delta\| <C \delta^\zeta,\quad
\|F(u_{n}) - f_\delta\|\ge C\delta^\zeta,\quad n<n_\delta ,\quad C>1,\quad \zeta \in(0,1).
\end{equation}
Integrals of the form 
$\int_0^1 e^{-|x-y|}h(y)dy$ in \eqref{1eq41} and \eqref{eq44} are computed by 
using
the trapezoidal rule. The noisy function, used in the test, is
$$
f_\delta(x) = f(x) + \kappa f_{noise}(x),\quad \kappa=\kappa(\delta)>0.
$$
The noise level $\delta$ and the relative noise level are defined by 
$$
\delta = \kappa\| f_{noise}(x)\|,\quad \delta_{rel}:=\frac{\delta}{\|f\|}.
$$
The constant $\kappa$ is computed in such a way that the relative noise level
$\delta_{rel}$ equals to some desired value, i.e.,
$$
\kappa = \frac{\delta}{\| f_{noise}(x)\|}=\frac{\delta_{rel}\|f\|}{\| f_{noise}\|}.
$$
We have used the relative noise level as an input parameter in the test.

In all figures the $x$-variable runs through the interval $[0,1]$,  
and the graphs represent the numerical solutions 
$u_{DSM}(x)$ and the exact solution $u_{exact}(x)$. 

As we have proved, the iterative scheme converges 
to the minimal-norm solution when
$a_n=\frac{d}{(5+hn)^b} $, $b\in(0,1)$, $\frac{10b}{5^b}\le d\le 2\times 5^{1-b}$ and $\gamma_n$ are "sufficiently" small. 
The choice of $\gamma_n$ depends on the problem one wants to solve because 
$\gamma_n$ depends on $\sigma_R$ which varies from problems to problems. 
Note that if one chooses $\gamma_n$ to be too small,
then one needs many iterations in order to reach the stopping
time $n_\delta$ in \eqref{eq53}. Consequently, the computation time will
be large in this case. For $\sigma$-inverse problems where the constant 
$\sigma = \sigma_R$ 
can be estimated then it is not difficult to choose $\gamma_n$ satisfying \eqref{eqk1}. 

In the numerical experiments we found that our method works well with $a(0)\in[0.1,1]$. In the
test we chose $a_n$ by the formula $a_n := \frac{a(0)}{(n+5)^{\zeta}}$ 
where $a(0) = 0.1$ and $\zeta=0.99$.
We carried out the experiments with $\gamma_n=h=const\in (0,1]$, and the method works well with this choice of $\gamma_n$.
If one chooses $h>0$ too small, then it takes more computer time for the method to converge.  
The number of node points, used in
computing integrals \eqref{1eq41} and \eqref{eq44}, was $N = 100$. 
In all the experiments, the exact solution is chosen as follows
$$
u_{exact}(x) = 
\left \{
\begin{matrix}
0 &\quad \text{if}\quad x\in[0,0.5)\\
1 &\quad \text{if}\quad x\in(0.5,1].
\end{matrix}
\right .
$$ 
As we have mentioned above, $F'(u)$ is not boundedly invertible in a neighborhood of 
$u_{exact}$. In particular, $F'(u_{exact})$ is not boundedly invertible. 
Thus, one can not use classical methods such as Newton's method or Gauss-Newton 
method to solve for $u_{exact}$. 

Numerical results for various values of $\delta_{rel}$ are presented in Table~\ref{table1}. 
From Table~\ref{table1} one can see that the number of iterations $n_\delta$ tends to 
go to $\infty$ as $\delta$ goes to $0$. Numerical experiments showed that 
$n_\delta\to \infty$ as $\delta\to 0$. 
Note that our choice of $a(t)$ in this experiment does not satisfy 
condition \eqref{lilo} which is a sufficient condition for having $n_\delta\to \infty$ as $\delta\to 0$. 
Table~\ref{table1} shows that the iterative scheme yields good numerical results. 
\begin{table}[ht] 
\caption{Results when $a(0)=0.1$ and $h=1$.}
\label{table1}
\centering
\small
\begin{tabular}{|@{  }c@{\hspace{2mm}}
@{\hspace{2mm}}|c@{\hspace{2mm}}|c@{\hspace{2mm}}|c@{\hspace{2mm}}|c@{\hspace{2mm}}|
c@{\hspace{2mm}}|c@{\hspace{2mm}}|c@{\hspace{2mm}}r@{\hspace{2mm}}l@{}} 
\hline
$\delta_{rel}$       &0.05   &0.03   &0.02    &0.01    &0.003 &0.001\\
\hline
Number of iterations &5	 &6     &8      &13   &39    &104\\
\hline 
$\frac{\|u_{DSM} - u_{exact}\|}{\|u_{exact}\|}$&0.166	 &0.111   &0.108    &0.076   &0.065    &0.045\\
\hline
\end{tabular}
\end{table}

Figure~\ref{figmono} plots the numerical results when relative noise levels are
$\delta_{rel}=0.01$ and $\delta_{rel} = 0.001$. 
The noise function in this example is a 
normally distributed random vector of length $N$ 
with mean 0 and variance 1. Here $N$ is the number of nodal points used in discretizing the interval $[0,1]$.   

\begin{figure}[!h!tb]
\centerline{%
\includegraphics[scale=0.83]{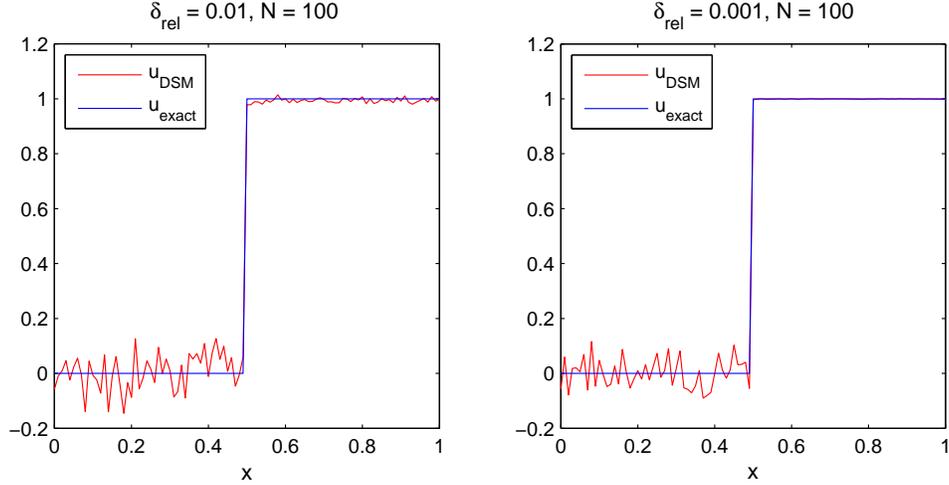}}
\caption{Plots of solutions obtained by the iterative scheme when $N = 100$, 
$\delta_{rel}=0.01$ (left) 
and $\delta_{rel}=0.001$ (right).}
\label{figmono}
\end{figure}

Figure~\ref{figmono2} plots the numerical results when
the noise levels are $\delta_{rel}=0.01$ and $\delta_{rel} = 0.001$. 
In this experiment we choose the 
noise function by the formula $f_{noise}(x)=\sin(3\pi x),x\in[0,1]$. 

\begin{figure}[!h!tb]
\centerline{%
\includegraphics[scale=0.83]{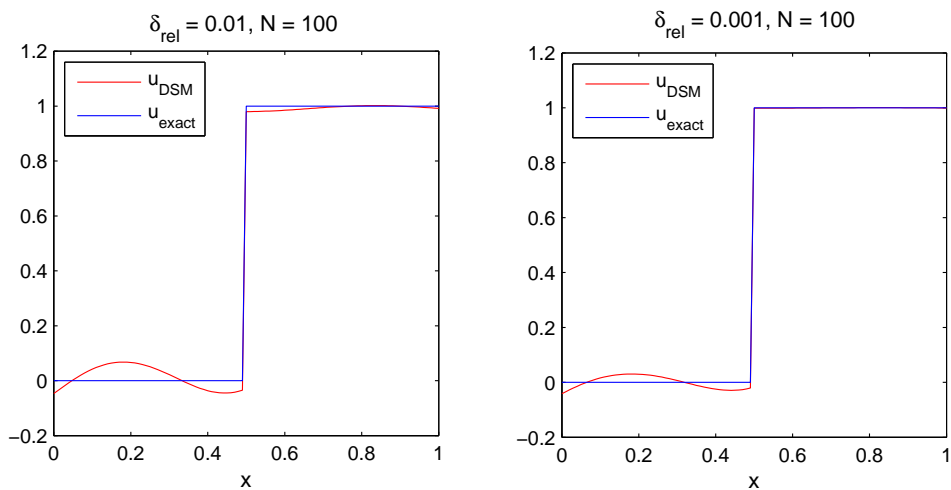}}
\caption{Plots of solutions obtained by the iterative scheme when $N = 100$, 
$\delta_{rel}=0.01$ (left) 
and $\delta_{rel}=0.001$ (right).}
\label{figmono2}
\end{figure}

In computations the functions $u,f$ and $f_\delta$ are vectors in $\mathbb{R}^N$
where $N$ is the number of nodal points. The norm used in computations is the Euclidean length
or $L^2$ norm of 
$\mathbb{R}^N$.

We have also carried out numerical experiments with 
$a_n=\frac{10}{(5+n)^{0.99}}$. For this choice of $a_n$ the convergence 
of $u_{n_\delta}$ to the unique solution of the problem is guaranteed by 
Theorem~\ref{mainthm}--\ref{theorem3}. However, the numerical experiment 
showed that using this choice of $a_n$ does not bring any improvement in 
accuracy while requiring more time for computation. Experiments also 
showed that for this problem it is better to use $a_n = 
\frac{a(0)}{(5+n)^{0.99}}$ with $a(0)\in [0.1,1]$.

From the numerical results we conclude that the proposed stopping rule yields 
good results in this problem.


\begin{thebibliography}{99}

\bibitem{BG}
A. Bakushinskii and A. Goncharskii, Ill-Posed Problems: Theory and Applications, Dordrecht, Kluwer, 1994.

\bibitem{D}
K. Deimling, Nonlinear functional analysis, Springer-Verlag, Berlin, 1985.

\bibitem{546} N. S. Hoang and A. G. Ramm,
An iterative scheme for solving
nonlinear equations with monotone operators. BIT, 48, N4, (2008), 725-741.

\bibitem{R549} N. S. Hoang and A. G. Ramm,
 Dynamical Systems Gradient method for solving nonlinear
equations with monotone operators, Acta Appl. Math., 106, (2009) , 473-499.

\bibitem{R550} N. S. Hoang and A. G. Ramm,
A new version of the Dynamical systems method (DSM)
for solving nonlinear quations with monotone operators, Diff. Eq. Appl., 1, N1, (2009), 1-25.

\bibitem{R544}
Hoang, N.S. and Ramm, A. G., Dynamical systems method for solving nonlinear equations with
monotone operators, Math. Comp., 79, (2010), 239-258.

\bibitem{I}
V. Ivanov, V. Tanana and V. Vasin, Theory of ill-posed problems, VSP, Utrecht, 2002.

\bibitem{M}
V. A. Morozov, Methods of solving incorrectly posed problems, Springer-Verlag, New York, 
1984.

\bibitem{R499} A. G. Ramm, Dynamical systems method for solving
operator equations, Elsevier, Amsterdam, 2007.


\bibitem{R452} A. G. Ramm, 
Global convergence for ill-posed equations
with monotone operators: the dynamical systems method, J.
Phys A, 36, (2003), L249-L254.

\bibitem{454} A. G. Ramm,
 Dynamical systems method for solving nonlinear
operator equations, International Jour. of
Applied Math. Sci., 1, N1, (2004), 97-110.


\bibitem{485} A. G. Ramm, 
Dynamical systems method (DSM) and
nonlinear problems, in the book: Spectral Theory and Nonlinear 
Analysis,
World Scientific Publishers, Singapore, 2005, 201-228. (ed J.
Lopez-Gomez).


\bibitem{491} A. G. Ramm,
Dynamical systems method (DSM) for unbounded
operators, Proc. Amer. Math. Soc., 134, N4, (2006), 1059-1063.

\bibitem{Tau} U. Tautenhahn, 
On the method of Lavrentiev regularization for
nonlinear ill-posed problems, Inverse Problems 18 (2002) 191–207.

\bibitem{VA} V. V. Vasin and A. L. Ageev, Ill-Posed Problems with a Priori Information, Utrecht, VSP, 1995.

\end{thebibliography}
\end{document}